%% file: CSP.tex
\documentclass[10pt]{amsart}
\usepackage{amsmath,amssymb,amsthm,graphicx,enumerate,color,bbm}

\newcommand\NN{\mathbb{N}}
\newcommand\QQ{\mathbb{Q}}

\theoremstyle{plain}% default
\newtheorem{thm}{Theorem}[section]
\newtheorem{prop}[thm]{Proposition}
\newtheorem{lemma}[thm]{Lemma}

\newtheorem{conj}[thm]{Conjecture}
\newtheorem*{LI}{Lagrange inversion}

\begin{document}%%%%%%%%%%%%%%%%%%%%%%%%%%%%%%%%%%
%%%%%%%%%%%%%%%%%%%%%%%%%%%%%%%%%%%%%%%%%%%

\title{Cyclic Sieving Phenomenon in Non-crossing Connected Graphs}
\author{Alan Guo}
\address{Department of Mathematics\\Duke University\\Durham, NC 27708}
\email{alan.guo@duke.edu}

\begin{abstract}
We prove an instance of the cyclic sieving phenomenon in non-crossing
connected graphs, as conjectured by S.-P. Eu.
\end{abstract}

\date{July 25, 2010}

\maketitle

%%%%%%%%%%%%%%%%%%%%%%%%%%%%%%%%%%%%%%%%%%%
\section{Introduction}%%%%%%%%%%%%%%%%%%%%%%%%%%%%%%%%
%%%%%%%%%%%%%%%%%%%%%%%%%%%%%%%%%%%%%%%%%%%

A \emph{non-crossing graph} on a finite set $S$ is a graph with vertices
indexed by $S$ arranged in a circle such that no edges cross. When we
say a graph on $n$ vertices, we will mean $S = \{1,\ldots,n\}$.
In~\cite{flajoletnoy}, Flajolet and Noy showed that the number $c_{n,k}$
of non-crossing connected graphs (see Figure~\ref{f:graph}) on $n$ vertices
with $k$ edges, $n-1 \le k \le 2n-3$, is
\begin{equation}\label{eq:C}
c_{n,k} = \frac{1}{n-1} {3n-3 \choose n+k}{k-1 \choose n-2}.
\end{equation}
\begin{figure}[htbp]
\begin{center}
\input{graph.pstex_t}
\caption{A non-crossing connected graph on $12$ vertices with $14$ edges}
\label{f:graph}
\end{center}
\end{figure}
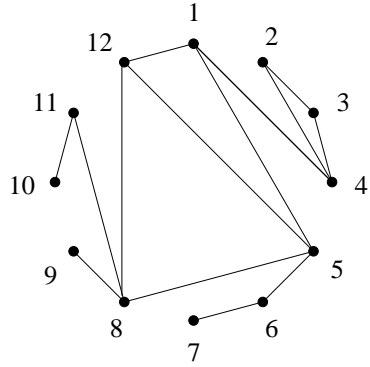

Define
$$
{n \brack k}_q = \frac{[n]!_q}{[k]!_q [n-k]!_q}
$$
where $[n]!_q = [n]_q [n-1]_q \cdots [1]_q$ and $[n]_q = 1 + q + \cdots + q^{n-1}
= \frac{1-q^n}{1-q}$.
The formula in \eqref{eq:C} admits a natural $q$-analogue:
\begin{equation}\label{eq:Cq}
c(n,k;q) = \frac{1}{[n-1]_q}{3n-3 \brack n+k}_q {k-1 \brack n-2}_q.
\end{equation}
It turns out that $c(n,k;q)$ is a polynomial in $q$, with nonnegative integer
coefficients; see Proposition~\ref{p:polynomial} below.

The main result of this paper is the following, which was conjectured by
S.-P. Eu~\cite{Eu}.
\begin{thm}\label{t:main}
Let $n \ge 1$ and $n-1 \le k \le 2n-3$, and let $X$ be the set of non-crossing
connected graphs on $n$ vertices with $k$ edges.
If $d \ge 1$ divides $n$ and $\omega$ is a primitive $d$-th root of
unity, then
$$
c(n,k;\omega) = s_d(n,k)
$$
where we define
$$
s_d(n,k) = \#\left\{ x \in X : \text{$x$ is fixed under rotation by $\frac{2\pi}{d}$}\right\}.
$$
\end{thm}

In \cite{RSW}, Reiner, Stanton, and White introduced the notion of the
cyclic sieving phenomenon. A triple $(X,X(q),C)$ consisting of a finite set
$X$, a polynomial $X(q) \in \NN[q]$ satisfying $X(1) = |X|$,
and a cyclic group $C$ acting on $X$ exhibits the \emph{cyclic sieving
phenomenon} if, for every $c \in C$, if $\omega$ is a primitive root of unity
of the same multiplicative order as $c$, then
$$
X(\omega) = \#\{x \in X : c(x) = x \}.
$$
In \eqref{eq:C}, the two extreme cases, $k = n-1$ and $k = 2n-3$, correspond to
non-crossing spanning trees and $n$-gon triangulations respectively.
In the former case, Eu and Fu showed in \cite{EuFu} that quadrangulations
of a polygon exhibit the cyclic sieving phenomenon, where the cyclic
action is cyclic rotation of the polygon, and they showed a bijection
between quadrangulations of a $2n$-gon with non-crossing spanning
trees on $n$ vertices.
\begin{figure}[htbp]
\input{eubijection.pstex_t}
\caption{Bijection between a spanning tree on a $5$ vertices and a
quadrangulation of a $10$-gon.}
\label{fig:eubijection}
\end{figure}
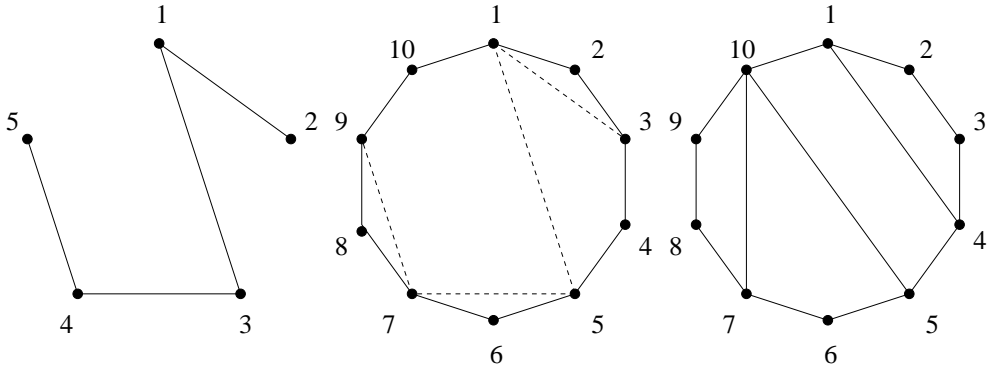
The bijection mapping is as follows: given a non-crossing spanning tree on
$n$ vertices, for each edge connecting $i$ to $j$, draw a dotted line from
$2i-1$ to $2j-1$ in a $2n$-gon. Then the quadrangulation of this $2n$-gon
is defined by quadrangles whose diagonals are the dotted lines; conversely,
given a $2n$-gon, every quadrangle has a diagonal whose endpoints are
odd numbers, so we may perform the reverse procedure to get an inverse
mapping (see Figure~\ref{fig:eubijection}). This bijection preserves the cyclic
sieving phenomenon, since rotation by $\frac{2\pi}{n}$ in the tree corresponds to
rotation by $\frac{\pi}{n}$ in the $2n$-gon.

In the latter case, Reiner, Stanton, and White showed
in \cite{RSW} that polygon dissections of a polygon exhibit the cyclic
sieving phenomenon where the cyclic action is also rotation. In particular,
triangulations acted upon by rotations exhibit the cyclic sieving phenomenon.
These results inspired Eu to conjecture Theorem~\ref{t:main}, which
we prove in the following sections.
The case $d = 1$ in Theorem~\ref{t:main} follows from \eqref{eq:C}.
We therefore consider the following three cases: $d = 2$ and $k$ is odd,
$d = 2$ and $k$ is even, and $d \ge 3$.

%%%%%%%%%%%%%%%%%%%%%%%%%%%%%%%%%%%%%%%%%%%
\section{Lagrange Inversion}%%%%%%%%%%%%%%%%%%%%%%%%%%%%%
%%%%%%%%%%%%%%%%%%%%%%%%%%%%%%%%%%%%%%%%%%%

In the following sections, we will use the Lagrange-B\"urmann inversion
theorem to extract coefficients of certain generating functions. If
$\phi(z) \in \QQ[[z]]$, then we define $[z^n]\phi(z)$ to be the coefficient
of $z^n$ in $\phi(z)$.

\begin{LI}
Let $\phi(u) \in \QQ[[u]]$ be a formal power series with $\phi(0) \ne 0$, and let $y(z)
\in \QQ[[z]]$ satisfy $y = z\phi(y)$. Then, for an arbitrary series $\psi$, the coefficient of
$z^n$ in $\phi(y)$ is given by
$$
[z^n]\psi(y(z)) = \frac{1}{n} [u^{n-1}]\phi(u)^n\psi'(u).
$$
\end{LI}
Lagrange inversion may be applied to bivariate generating functions by
treating the second variable as a parameter.

We begin by illustrating how Flajolet and Noy used Lagrange inversion to
find \eqref{eq:C}. Let $C(z,w)$ be the generating function for $c_{n,k}$, that is,
$$
C(z,w) = \sum_{n,k} c_{n,k} z^nw^k.
$$
Then it can be shown using a combinatorial argument that $C$ satisfies
$$
wC^3 + wC^2 - z(1+2w)C + z^2(1+w) = 0.
$$
Setting $C = z + zy$, this becomes
$$
wz(1+y)^3 = y(1-wy)
$$
which can be put in the Lagrange form
\begin{equation}\label{eq:y}
y = z\frac{w(1+y)^3}{1-wy}.
\end{equation}
The result \eqref{eq:C} then follows upon application of Lagrange inversion
on $y$. We will in fact use this same function $y$ multiple times in our proofs.

%%%%%%%%%%%%%%%%%%%%%%%%%%%%%%%%%%%%%%%%%%%
\section{The case where $d=2$ and $k$ is odd}%%%%%%%%%%%%%%%%%%%
%%%%%%%%%%%%%%%%%%%%%%%%%%%%%%%%%%%%%%%%%%%

In this section, we prove that Theorem~\ref{t:main} holds when $d = 2$
and $k$ is odd.
Recall that $d$ divides~$n$, so $n$ must be even in this case. The case where $n=2$
is trivial since there is only $1$ non-crossing connected graph on $2$ vertices,
so we may assume that $n > 2$. For this section, define $n' = \frac n2$ and
$k' = \frac{k+1}{2}$. It is a straightforward computation to verify that
\begin{equation}\label{eq:kodd}
c(n,k;-1) %= {\frac{3n-4}{2} \choose \frac{n+k-1}{2}}{\frac{k-1}{2} \choose \frac{n-2}{2}}
= {3n'-2 \choose n'+k'-1}{k'-1 \choose n'-1}.
\end{equation}
Define
$$
f_{n,k} = \#\{x \in X : \text{$x$ has an edge from $1$ to $n$}\}.
$$
\begin{lemma}
With $f_{n,k}$ defined as above, we have
$$
s_2(n,k) = n' \cdot f_{n' + 1, k'}.
$$
\end{lemma}
\begin{proof}
Given a centrally symmetric with an odd number of edges, exactly
one of the edges must be a diameter. There are $n'$ choices
for the diameter. Once a diameter has been fixed, the remaining $k-1$ edges
are determined by the $k'-1$ edges on either side of the diameter.
\begin{figure}[htbp]
\input{graphkodd.pstex_t}
\ \ \ \ \ \ \ \
\input{graphkodd2.pstex_t}
\caption{The bijection between centrally symmetric $n$-vertex, $k$-edge
graph with fixed diameter and $(\frac{n}{2}+1)$-vertex, $\frac{k+1}{2}$-edge
graph with edge between $1$ and $\frac{n}{2}+1$.}
\label{f:graphkodd}
\end{figure}
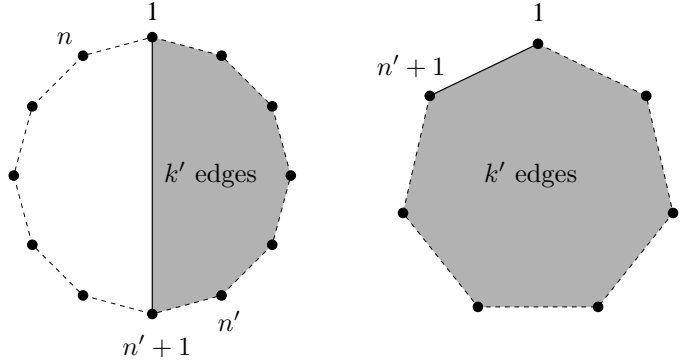
Without loss of generality, assume the diameter has endpoints $1$ and $n'+1$.
Then we have a bijection (see Figure~\ref{f:graphkodd})
between the graphs we wish to count and graphs on $\left\{1,\ldots,n'+1\right\}$
with $k'$ edges including the edge from $1$ to $n'+1$. This
is counted by $f_{n'+1, k'}$.
\end{proof}

We define some more notation. Recall that $c_{n,k} = |X|$. Define $d_{n,k}$
to be the number of non-crossing graphs on $\{1,\ldots,n\}$ with $k$ edges and
exactly two connected components such that $1$ and $n$ are in different components.

\begin{lemma}
With $d_{n,k}$ defined above, we have
$$
d_{n,k} = \frac{2}{n-2}{3n-5 \choose n+k}{k-1 \choose n-3}.
$$
\end{lemma}
\begin{proof}
Let $D(z,w) = \sum d_{n,k}z^nw^k$ and let $C(z,w) = \sum c_{n,k}z^nw^k$.
Since $d_{n,k}$ counts graphs with two connected components, which are
each counted by $c_{n,k}$, we therefore have $D = C^2$. To find the coefficient
of $z^nw^k$, we use Lagrange inversion. Recall from \eqref{eq:y} that
$y = z \frac{w(1+y)^3}{1-wy}$. But
$D = C^2 = z^2 + z^2(y^2 + 2y)$. Therefore
$$
[z^nw^k]D = [z^{n-2}w^k] y^2 + 2[z^{n-2}w^k]y.
$$
Computing each of these separately, we have
\begin{eqnarray*}
[z^{n-2}w^k]y &=& \frac{1}{n-2}[u^{n-3}w^k] \frac{w^{n-2}(1+u)^{3n-6}}{(1-uw)^{n-2}} \\
&=& \frac{1}{n-2}[u^{n-2}w^{k-n+2}] \frac{(1+u)^{3n-6}}{(1-uw)^{n-2}} \\
&=& \frac{1}{n-2} (-1)^{k-n+2} {2-n \choose k-n+2} [u^{n-3}] u^{k-n+2} (1+u)^{3n-6} \\
&=& \frac{1}{n-2} {k-1 \choose k-n+2}[u^{n-3}] u^{k-n+2}(1+u)^{3n-6} \\
&=& \frac{1}{n-2} {k-1 \choose n-3} [u^{2n-k-5}] (1+u)^{3n-6} \\
&=& \frac{1}{n-2} {3n-6 \choose n+k-1} {k-1 \choose n-3}
\end{eqnarray*}
and
\begin{eqnarray*}
[z^{n-2}w^k]y^2 &=& \frac{2}{n-2} [u^{n-4}w^k] \frac{w^{n-2}(1+u)^{3n-6}}{(1-uw)^{n-2}} \\
&=& \frac{2}{n-2} [u^{n-4}w^{k-n+2}] \frac{(1+u)^{3n-6}}{(1-uw)^{n-2}} \\
&=& \frac{2}{n-2} (-1)^{k-n+2} {2-n \choose k-n+2} [u^{n-4}] u^{k-n+2} (1+u)^{3n-6} \\
&=& \frac{2}{n-2} {k-1 \choose k-n+2} [u^{2n-k-6}] (1+u)^{3n-6} \\
&=& \frac{2}{n-2} {3n-6 \choose n+k} {k-1 \choose n-3}.
\end{eqnarray*}
where we have used the identity
$$
{n \choose k} = (-1)^k{k-n-1 \choose k}.
$$
The result then follows from Pascal's identity.
\end{proof}

\begin{lemma}
The sequence $f_{n,k}$ satisfies the recurrence
$$
f_{n,k} + f_{n,k+1} = c_{n,k} + d_{n,k}
$$
with the base case
$$
f_{n,2n-3} = c_{n,2n-3} = \frac{1}{n-1}{2n-4 \choose n-2}.
$$
\end{lemma}
\begin{proof}
The base case follows from the fact that every triangulation must contain the
edge from $1$ to $n$. Now consider a non-crossing connected graph with $k+1$ edges
on $\{1,\ldots,n\}$ with the edge $1$ to $n$. We have two cases. When we remove
this edge, either the remaining graph is connected or not. If the remaining graph
is connected, then we have a non-crossing connected graph with $k$ edges
without the edge from $1$ to $n$. This is counted by $c_{n,k} - f_{n,k}$.
If the remaining graph is not connected, then there are exactly two connected
components, and $1$ and $n$ lie in separate components. This is
counted by $d_{n,k}$. Hence
$$
f_{n,k+1} = c_{n,k} + d_{n,k} - f_{n,k}.
$$
\end{proof}

As a corollary to this lemma, it follows that one has the recurrence
\begin{equation}\label{eq:recurrence}
s_2(2n-2,2k-1) + s_2(2n-2,2k+1) = (n-1)c_{n,k} + (n-1)d_{n,k}
\end{equation}
with base case
$$
s_2(2n-2,4n-7) = (n-1)c_{n,2n-3} = {2n-4 \choose n-2}.
$$
To show that $c(n,k;-1) = s_2(n,k)$ for even $n$ and odd $k$ or, equivalently,
$c(2n-2,2k-1;-1) = s_2(2n-2,2k-1)$ for any positive integers $n>2$ and
$n-1 \le k \le 2n-3$, it suffices to show that $c(2n-2,2k-1;-1)$ satisfies the same
recurrence \eqref{eq:recurrence} as $s_2(2n-2,2k-1)$. The base case is immediate:
$$
c(2n-2,4n-7;-1) = {3n-5 \choose 3n-5}{2n-4 \choose n-2} = {2n-4 \choose n-2}.
$$
We now show that $c(2n-2,2k-1;-1)$ satisfies the recurrence relation as well,
which completes the proof that the theorem holds for $d = 2$ and odd $k$.
\begin{prop}
$c(2n-2,2k-1;-1)$ satisfies
$$
c(2n-2,2k-1;-1) + c(2n-2,2k+1;-1) = (n-1)c_{n,k} + (n-1)d_{n,k}.
$$
\end{prop}
\begin{proof}
From \eqref{eq:kodd}, we see that all we need to verify is
$$
{3n-5 \choose n+k-2}{k-1 \choose n-2} + {3n-5 \choose n+k-1}{k \choose n-2}
= {3n-3 \choose n+k}{k-1 \choose n-2} + \frac{2n-2}{n-2}{3n-5 \choose n+k}
{k-1 \choose n-3},
$$
which we leave as a straightforward exercise for the reader.
\end{proof}

%%%%%%%%%%%%%%%%%%%%%%%%%%%%%%%%%%%%%%%%%%%
\section{The case where $d=2$ and $k$ is even}%%%%%%%%%%%%%%%%%%%
%%%%%%%%%%%%%%%%%%%%%%%%%%%%%%%%%%%%%%%%%%%

In this section, we prove that Theorem~\ref{t:main} holds when $d = 2$
and $k$ is even.
As in the previous case, it is again a straightforward computation to verify
that
$$
c(n,k;-1) = {\frac{3n-4}{2} \choose \frac{n+k}{2}}{\frac{k-2}{2} \choose \frac{n-2}{2}}.
$$

Let $a_{n,k}$ denote the number of non-crossing connected graphs with
$2n$ vertices and $k$ pairs of antipodal edges, where a diameter counts
as one pair. Then
$$
a_{n,k} = s_2(2n,2k-1) + s_2(2n,2k) = c(2n,2k-1;-1) + s_2(2n,2k).
$$
Therefore, to show Theorem~\ref{t:main} holds for $d=2$ and even $k$,
it suffices to show that
\begin{equation}\label{eq:keven}
a_{n,k} = c(2n,2k-1;-1) + c(2n,2k;-1) = {3n-1 \choose n+k}{k-1 \choose n-1}.
\end{equation}
Let $F$ be the generating function for $f_{n,k}$, i.e. $F(z,w) = \sum f_{n,k}z^nw^k$.
Similarly, let $A$ be the generating function for $a_{n,k}$.

\begin{lemma}
$$
a_{n,k} = \sum_{m=1}^n \  \sum_{k_1 + \cdots + k_m = k} \ 
\sum_{1 \le v_1 < \cdots < v_m \le n} \  \prod_{i=1}^m f_{v_{i+1} - v_i +1,k_i}
$$
where $v_{m+1} = v_1 + n$.
\end{lemma}
\begin{proof}
Consider a non-crossing connected graph with $2n$ vertices and $k$ pairs
of antipodal edges. There exists a unique positive integer $m$ such that the
center of the $2n$-gon lies inside a $2m$-gon formed by edges of the graph
and such that no other edges lie inside the $2m$-gon. This $m$ is at most
$n$. Now, exactly $m$ of the vertices of this $2m$-gon, call them
$v_1 < \cdots < v_m$, lie in the set $\{1,\ldots,n\}$ due to the antipodal condition
on the edges.
\begin{figure}[htbp]
\input{graphkeven.pstex_t}
\caption{A graph with an inner $2m$-gon, where $m = 4$.}
\label{f:graphkeven}
\end{figure}
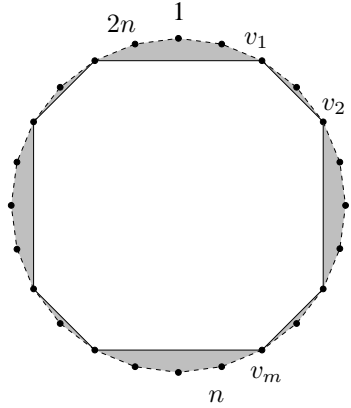
All edges not used in the $2m$-gon lie outside of it (see Figure~\ref{f:graphkeven}).
The $(m+1)$-th vertex
is antipodal to $v_1$, hence $v_{m+1} = v_1 + n$. For each $i$, there is an edge
from $v_i$ to $v_{i+1}$ and $k_i-1$ other edges on the vertices
$\{v_i, v_i + 1, \ldots, v_{i+1}\}$, such that $k_1 + \cdots + k_m = k$. Such
a graph is counted by $f_{v_{i+1} - v_i+1,k_i}$. Thus we get the corresponding
sum.
\end{proof}

\begin{lemma}
With $A$ and $F$ as defined above, we have
$$
\frac{A}{z} = \frac{\partial(F/z) / \partial z}{1 - F/z}.
$$
\end{lemma}
\begin{proof}
We show that
$$
a_{n,k} = \sum_{m=1}^n \  \sum_{k_1 + \cdots k_m = k} \ 
\sum_{n_1 + \cdots + n_m = n+m}(n_m-1)f_{n_m,k_m} \prod_{i=1}^{m-1} f_{n_i,k_i}.
$$
In the sum in the previous lemma, the term $\prod_{i=1}^m f_{v_{i+1} - v_i+1,k_i}$
is counted multiple times with the product written in this order. We show that it is
counted exactly $n + v_1 - v_m$ times. Consider any $m$-element subset
$\{v_1,\ldots,v_m\} \subseteq \{1,\ldots,n\}$ with $v_1 < \cdots < v_m$. For $j =1,\ldots,v_1-1$, this subset yields the same summand as $\{v_1 - j, \ldots, v_m - j\}$.
Therefore, we can identify any subset $\{v_1,\ldots,v_m\}$ with
$\{1,\ldots,v_m - v_1 + 1\}$. There are exactly $n + v_1 - v_m$ subsets
corresponding to this one, each with largest element $v_m - v_1 + 1, v_m - v_1 + 2,
\ldots, n$. This proves the sum identity above.

For the equality of generating functions, we insert variables into the above identity:
\begin{equation}\label{eq:triplesum}
a_{n,k}z^nw^k = \frac{1}{z^{m-2}} \sum_{m=1}^n \ \sum_{k_1 + \cdots k_m = k} \
\sum_{n_1 + \cdots + n_m = n+m} (n_m-1)f_{n_m,k_m}z^{n_m-2}w^{k_m}
\prod_{i=1}^{m-1} f_{n_i,k_i}z^{n_i}w^{k_i}.
\end{equation}
We note that
$$
\frac{\partial (F/z)}{\partial z} = \sum_{n,k} (n-1) f_{n,k} z^{n-2}w^k
$$
so, summing over all $n$ and $k$ in \eqref{eq:triplesum}, we get
$$
A = \frac{\partial (F/z)}{\partial z} \left( z + F + \frac{F^2}{z} + \frac{F^3}{z^2} + \cdots \right)
= z \frac{\partial (F/z)}{\partial z} \left( \frac{1}{1 - F/z} \right).
$$
\end{proof}

\begin{prop}
$$
a_{n,k} = {3n-1 \choose n+k}{k-1 \choose n-1}.
$$
\end{prop}
\begin{proof}
Let $H = F/z$ and let $C$ be the generating function for $c_{n,k}$ as in the
previous section and let $C = z + zy$. From the recurrence
$$
f_{n,k} + f_{n,k+1} = d_{n,k} + c_{n,k}, \text{ \ } n \ge 2, \text{ \ }
$$
and
$$
f_{1,k} = 0
$$
we have
$$
\left(1 + \frac1w\right)F = D + C - z = z^2(1+y)^2 + zy.
$$
Therefore, after some substitution and simplification, applying
the identity in \eqref{eq:y}, we get
$$
1 - H = \frac{1}{1+y}.
$$
From
$$
\frac{A}{z} = \frac{\partial H/\partial z}{1-H}
$$
we get
$$
\int \frac{A}{z} dz = \int \frac{dH}{1-H}
$$
or equivalently
$$
\sum_{n,k} \frac1n a_{n,k}z^nw^k = -\log(1 - H) = \log(1+y).
$$
By Lagrange inversion,
\begin{eqnarray*}
\frac{1}{n} a_{n,k} &=& [z^nw^k] \int \frac{A}{z}\,dz \\
&=& [z^nw^k] \log(1+y) \\
&=& \frac{1}{n} [u^{n-1}w^k] \frac{w^n(1+u)^{3n}}{(1- uw)^n} \frac{1}{1+u} \\
&=& \frac{1}{n} [u^{n-1}w^{k-n}] \frac{(1+u)^{3n-1}}{(1-uw)^n} \\
&=& \frac{1}{n} (-1)^{k-n} {-n \choose k-n} [u^{n-1}] u^{k-n} (1+u)^{3n-1} \\
&=& \frac{1}{n} {k-1 \choose n-1} [u^{2n-k-1}] (1+u)^{3n-1} \\
&=& \frac{1}{n} {3n-1 \choose n+k}{k-1 \choose n-1}
\end{eqnarray*}
whence our desired result.
\end{proof}

Comparing with \eqref{eq:keven} shows that Theorem~\ref{t:main} holds when
$d = 2$ and $k$ is even.

%%%%%%%%%%%%%%%%%%%%%%%%%%%%%%%%%%%%%%%%%%%
\section{The case where $d \ge 3$}%%%%%%%%%%%%%%%%%%%%%%%%%
%%%%%%%%%%%%%%%%%%%%%%%%%%%%%%%%%%%%%%%%%%%

Finally, in this section, we prove that Theorem~\ref{t:main} holds
when $d \ge 3$. For this section, define $n' = \frac nd$ and $k' = \frac kd$.
Again, it is a straightforward computation to verify that if $d | k$, then
$$
c(n,k;\omega) %= {\frac{3n}{d}-1 \choose \frac{n+k}{d}}{\frac kd-1 \choose \frac nd-1}.
= {3n'-1 \choose n'+k'}{k'-1 \choose n'-1}.
$$
\begin{lemma}
If $d \ge 3$ does not divide $k$, then $c(n,k;\omega) = 0$, where
$\omega$ is a primitive $d$-th root of unity.
\end{lemma}
\begin{proof}
Suppose $k \equiv r \pmod{d}$, where $0 < r < d$. If $r > d-3$, then
${3n-3 \brack n+k}_q = 0$. Grouping terms, we have
$$
{3n-3 \brack n+k}_q
= \frac{\overbrace{[3n-3]_q \cdots}^{d-3}}{\underbrace{[n+k]_q \cdots}_r}
\frac{\overbrace{[3n-d]_q \cdots [2n-k-d+r+1]_q}^{n+k-r}}{\underbrace{[n+k-r]_q \cdots
[1]_q}_{n+k-r}} \frac{\overbrace{[2n-k-d+r]_q \cdots [2n-k-2]_q}^{r-d+3}}{}.
$$
The center block of $n+k-r$ ratios as well as the $d-3$ ratios to the left of it go
to some number in the limit as
$q \to \omega$, and none of the terms to the left vanish at $q = \omega$
since none are divisible by $d$.
However, $2n-k-d+r \equiv 0 \pmod{d}$ so $[2n-k-d+r]_{q = \omega} = 0$.
Since $d-r \le 2$, one has $2n-k-d+r \ge 2n-k-2$ so the $[2n-k-d+r]_q$ term actually
exists in the expression above. Similarly, if $r \le d-3$, then ${k-1 \brack n-2}_q = 0$.
Grouping terms again, we have
$$
{k-1 \brack n-2}_q
= \frac{\overbrace{[k-1]_q \cdots}^{r-1}}{\underbrace{[n-2]_q \cdots}_{d-2}}
\frac{\overbrace{[k-r]_q \cdots [k-r-n+d+1]_q}^{n-d}}{\underbrace{[n-d]_q \cdots
[1]_q}_{n-d}}
\frac{\overbrace{[k-r-n+d]_q \cdots [k-n+2]_q}^{d-r-1}}{}.
$$
As in the previous case, the center block of $n-d$ ratios as well as the $r-1$
ratios to the left of it go to some number in the limit as $q \to \omega$, and none
of the terms to the left vanish at $q = \omega$ since none are divisible by $d$.
If $r < d-3$, then $d-2 \ge r$, so $k-r-n+d \ge k-n+2$, hence the
$[k-r-n+d]_q$ term exists in the expression above and since $k-r-n+d \equiv 0
\pmod{d}$, it has a zero at $q = \omega$. If $r = d-3$, then there are
exactly $d-r-1=2$ terms in the right block, one of which is $[k-n+3]_q$.
Since $k-n+3 \equiv r+3 \equiv 0 \pmod{d}$, this term has a zero at $q = \omega$.
\end{proof}

If $d$ does not divide $k$, then in fact there are no graphs with $k$
edges that are fixed under rotation by $\frac{2\pi}{d}$, since
each edge lies in a free orbit under the action of rotation. We
henceforth assume that~$d | k$.
\begin{lemma}
$$
s_d(n,k) = %\frac{n}{d} \cdot f_{\frac nd+1, \frac kd} +
%s_2 \left( \frac{2n}{d}, \frac{2k}{d}\right).
n' \cdot f_{n'+1,k'} + s_2(2n',2k').
$$
\end{lemma}
\begin{proof}
If $\Gamma$ is a non-crossing connected graph on $\{1,\ldots,n\}$ fixed under
rotation by $\frac{2\pi}{d}$, then there are two cases: either the edges form a
central $d$-gon or not. In the former case, every edge is purely determined by
the edges on the first $n'+1$ vertices. In fact, there is bijection between
such graphs and non-crossing connected graphs on $n'+1$ vertices
with the edge from $1$ to $n'+1$. There are $f_{n'+1, k'}$
such graphs, and there are $n'$ possible $d$-gons. In the latter case,
the edges are determined by edges on the first $2n'$ vertices. We
construct a bijection between such graphs and centrally symmetric non-crossing
connected graphs on $2n'$ vertices with $2k'$ edges as follows
(see Figure~\ref{f:bijection}): 
\begin{figure}[htbp]
\begin{center}
\input{bijection.pstex_t}
\caption{The bijective construction when $(n,k,d) = (12,12,3)$}
\label{f:bijection}
\end{center}
\end{figure}
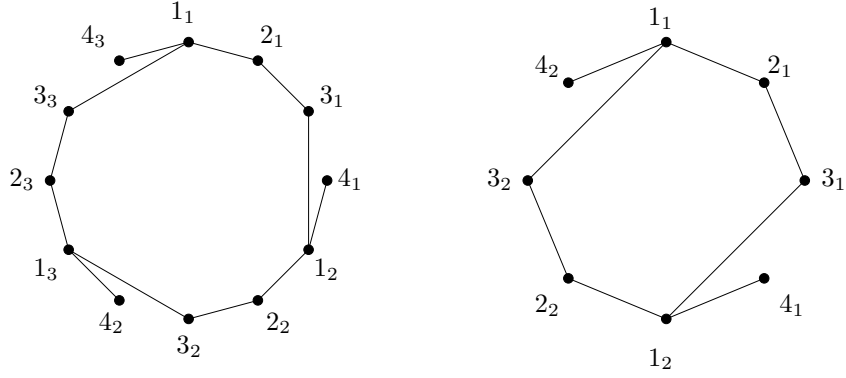
Going around clockwise in the graph,
label the first $m$ vertices $1_1,2_1,\ldots,n'_1$, label the next set of vertices
$1_2,2_2,\ldots,n'_2$, and so on. Construct a non-crossing graph with $2n'$
vertices labeled $1_1,2_1,\ldots,n'_1,1_2,2_2,\ldots,n'_2$.
For each edge from $i$ to $j$ in the original graph, we put an edge
with the same endpoints in the new graph. Finally, if there is an edge from
some $i_2$ to some $j_3$, we put an edge from $i_2$ to $j_1$ in the
new graph. This new graph therefore has $2k'$ edges.
It is straightforward to check that this is a bijection.
\end{proof}

\begin{prop}
For $d \ge 3$ and $\omega$ a primitive $d$th root of unity,
$$
c(n,k;\omega) = s_d(n,k).
$$
\end{prop}
\begin{proof}
By the previous Lemma, and results from the case where $d=2$,
\begin{eqnarray*}
s_d(n,k) &=& n' \cdot f_{n'+1, k'} + s_2\left( 2n', 2k' \right) \\
&=& {3n' - 2 \choose n'+k'}{ k' -1 \choose n' -1}
+ {3n' - 2 \choose n'+k'-1}{k' - 1 \choose n' - 1} \\
&=& \left[{ 3n' - 2 \choose n'+k'} +
{3n' - 2 \choose n'+k'-1} \right] {k' - 1 \choose n' - 1} \\
&=& {3n' - 1 \choose n'+k'}{ k' - 1 \choose n' - 1} \\
&=& c(n,k;\omega).
\end{eqnarray*}
\end{proof}

This completes the proof of Theorem~\ref{t:main}.

%%%%%%%%%%%%%%%%%%%%%%%%%%%%%%%%%%%%%%%%%%%
\section{Remarks and Future Work}%%%%%%%%%%%%%%%%%%%%%%%%%%
%%%%%%%%%%%%%%%%%%%%%%%%%%%%%%%%%%%%%%%%%%%

Recall in the definition of the cyclic sieving phenomenon given in the Introduction
that the function $X(q)$ must be a polynomial with nonnegative integer coefficients.
Looking at \eqref{eq:Cq}, it is not \emph{a priori} obvious that $c(n,k;q)$
is a polynomial with nonnegative integer coefficients. This is the content of
our next Proposition.
\begin{prop}\label{p:polynomial}
$c(n,k;q) \in \NN[q]$.
\end{prop}
\begin{proof}
To show that $c(n,k;q) \in \QQ[q]$, it suffices to show that in the expansion
$$
c(n,k;q) = \frac{\overbrace{[3n-3]_q \cdots [2n-k-2]_q}^A}
{\underbrace{[n+k]_q \cdots [1]_q}_B}
\frac{\overbrace{[k-1]_q \cdots [k-n+2]_q}^C}{\underbrace{[n-1]_q \cdots [2]_q}_D},
$$
for each $j \ge 1$, if $\omega$ is a primitive $j$-th root of unity, then
the order of the zero at $q = \omega$ is no smaller in the numerator than in the
denominator. For $j = 1$, it is clear that the numerator and denominator each
have a zero at $q = 1$ with multiplicity $2n+k-2$. For $j = 2$, we have
two cases. If $n$ is even, then block $B$ has $\lfloor\frac{n+k}{2}\rfloor$
zeros at $q = -1$ and block $D$ has $\frac{n-2}{2}$ zeros, whereas
block $A$ has at least $\lfloor\frac{n+k}{2}\rfloor$ zeros at $q = -1$
and block $C$ has at least $\frac{n-2}{2}$ zeros. If $n$ is even, we have two
subcases. If $k$ is odd, then the bottom has $\lfloor\frac{n+k}{2}\rfloor
+ \frac{n-1}{2}$ zeros while the the top has $\lceil\frac{n+k}{2}\rceil
+ \lfloor\frac{n-2}{2}\rfloor$, so they have the same order zero at $q = -1$.
If $k$ is even, then both the top and bottom have a zero at $q=-1$ of order
$\frac{n+k}{2} + \frac{n-1}{2}$. Now, if $j \ge 3$, we again have
two cases. If $n-1 \not\equiv 0 \pmod{j}$, then the bottom has a zero at
$q = \omega$ of order $\lfloor\frac{n+k}{j}\rfloor + \lfloor\frac{n-1}{j}\rfloor$
while block $A$ has a zero of order $\ge \lfloor\frac{n+k}{j}\rfloor$ and
block $C$ has a zero of order $\ge \lfloor\frac{n-1}{j}\rfloor$. If $n=1 \equiv 0
\pmod{j}$, then the bottom has a zero of order $\lfloor\frac{n+k}{j}\rfloor
+ \frac{n-1}{j}$ while, since $3n-3 \equiv 0 \pmod{j}$, the top has a zero of order
$\ge \lceil\frac{n+k}{j}\rceil + \lfloor\frac{n-2}{j}\rfloor$ which is
equal to $\lfloor\frac{n+k}{j}\rfloor + \lfloor\frac{n-2}{j}\rfloor + 1 =
\lfloor\frac{n+k}{j}\rfloor + \frac{n-1}{j}$.
It is well-known that $q$-binomial coefficients are polynomials in $\NN[q]$ with
symmetric unimodal coefficient sequences since, for example, ${n \brack k}_q$
is the generating function for integer partitions that fit in a $k \times (n-k)$ rectangle.
Therefore,
$$
{3n-3 \brack n+k}_q {k-1 \brack n-2}_q
$$
is a product of polynomials in $\NN[q]$ with symmetric unimodal coefficient sequences
and is thus itself a polynomial in $\NN[q]$ with a symmetric unimodal coefficient sequence.
The fact that $c(n,k;q) \in \NN[q]$ then follows from \cite[Proposition~10.1]{RSW}.
\end{proof}

\subsection{Cyclic sieving phenomenon in other types of graphs}
In \cite{flajoletnoy}, one finds various formulas for counting classes of
non-crossing graphs, of which \eqref{eq:C} is one. Consider the following
four formulae, found in \cite{flajoletnoy}:
\begin{eqnarray*}
T_n &=& \frac{1}{2n-1} {3n-3 \choose n-1} \\
C_{n,k} &=& \frac{1}{n-1} {3n-3 \choose n+k}{k-1 \choose k-n+1} \\
D_{n,k} &=& \frac{1}{k+1} {n-3 \choose k}{n+k-1 \choose k} \\
P_{n,k} &=& \frac{1}{n} {n \choose k}{n \choose k+1}
\end{eqnarray*}
where $T_n$ is the number of non-crossing trees on $n$ vertices, $C_{n,k}$
is the number of non-crossing connected graphs on $n$ vertices,
$D_{n,k}$ is the number of dissections of a convex $n$-gon using $k$
non-crossing diagonals, and $P_{n,k}$ is the number of non-crossing
partitions of size $n$ with $n-k$ blocks (for a definition of a non-crossing
partition, see \cite[Section~3]{flajoletnoy} or \cite[Section~7.2]{RSW}).
These equations all have natural $q$-analogues:
\begin{eqnarray*}
T(n;q) &=& \frac{1}{[2n-1]_q} {3n-3 \brack n-1}_q \\
C(n,k;q) &=& \frac{1}{[n-1]_q}{3n-3 \brack n+k}_q {k-1 \brack k-n+1}_q \\
D(n,k;q) &=& \frac{1}{[k+1]_q}{n-3 \brack k}_q {n+k-1 \brack k}_q \\
P(n,k;q) &=& \frac{1}{[n]_q} {n \brack k}_q {n \brack k+1}_q q^{k(k+1)}.
\end{eqnarray*}
Let $C$ be the cyclic group of order $n$ with a group action of rotation
on a non-crossing graph on $n$ vertices. As mentioned in the Introduction,
the collections of graphs for $T_n$, $D_{n,k}$, and $P_{n,k}$ exhibit
the sieving phenomenon with respect to their respective $q$-analogues
and $C$. Theorem~\ref{t:main} asserts that this is also true for
those graphs counted by $C_{n,k}$ and $X(q) = C(n,k;q)$. However, there
are more formulas in \cite{flajoletnoy}:
\begin{eqnarray*}
F_{n,k} &=& \frac{1}{2n-k} {n \choose k-1}{3n-2k-1 \choose n-k} \\
G_{n,k} &=& \frac{1}{n-1} \sum_{j=0}^{n-2} {n-1 \choose k-j}{n-1 \choose j+1}
{n-2+j \choose n-2}
\end{eqnarray*}
where $F_{n,k}$ is the number of non-crossing forests on $n$ vertices with $k$
components and $G_{n,k}$ is the number of non-crossing (not necessarily
connected) graphs on $n$ vertices with $k$ edges. These formulae admit
$q$-analogues as well:

\begin{conj}
Let $X$ be the set of non-crossing forests on $n$ vertices with $k$ components.
Let
$$
X(q) = \frac{1}{[2n-k]_q} {n \brack k-1}_q {3n-2k-1 \brack n-k}_q
$$
and let $C$ be the cyclic group of order $n$ acting on $X$ by rotation. Then
$(X,X(q),C)$ exhibits the cyclic sieving phenomenon.
\end{conj}

\begin{conj}
Let $X$ be the set of non-crossing graphs on $n$ vertices with $k$ edges.
Let
$$
X(q) = \frac{1}{[n-1]_q} \sum_{j=0}^{n-2} {n-1 \brack k-j}_q {n-1 \brack j+1}_q
{n-2+j \brack n-2}_q q^{j(j+n-k+2)}
$$
and let $C$ be the cyclic group of order $n$ acting on $X$ by rotation. Then
$(X,X(q),C)$ exhibits the cyclic sieving phenomenon.
\end{conj}

\subsection{Unifying algebraic proof of cyclic sieving phenomenon in graphs}
So far, separate combinatorial proofs have been offered in \cite{EuFu}, \cite{RSW}, and
this paper for the exhibition of the cyclic sieving phenomenon in the context
of the aforementioned classes of non-crossing graphs. It would be interesting
to see an algebraic proof of Theorem~\ref{t:main} along the lines of~
\cite[Proposition~2.1]{RSW}.

\section*{Acknowledgements}
This research was carried out in a summer REU at the University of Minnesota,
mentored by Vic Reiner and Dennis Stanton, and financially supported by
NSF grant DMS-1001933.
The author thanks Vic Reiner and Dennis Stanton for introducing him to
this fascinating problem, as well as for their guidance, support, and countless insightful
comments and suggestions. The author also thanks Jia Huang for carefully reading drafts
of this paper and for his helpful suggestions.

%%%%%%%%%%%%%%%%%%%%%%%%%%%%%%%%%%%%%%%%%%%
%%%%%%%%%%%%%%%%%%%%%%%%%%%%%%%%
%%%%%%%%%%%%%%%%%%%%%%%%%%%%%%%%%%%%%%%%%%%

%%%%%%%%%%%%%%%%%%%%%%%%%%%%%%%%%%%%%%%%%%%
\end{document}

%% file: graph.pstex_t
\begin{picture}(0,0)%
\includegraphics{graph.pstex}%
\end{picture}%
\setlength{\unitlength}{1973sp}%
\begingroup\makeatletter\ifx\SetFigFont\undefined%
\gdef\SetFigFont#1#2#3#4#5{%
  \reset@font\fontsize{#1}{#2pt}%
  \fontfamily{#3}\fontseries{#4}\fontshape{#5}%
  \selectfont}%
\fi\endgroup%
\begin{picture}(4549,4530)(3961,-5626)
\end{picture}%

%% file: eubijection.pstex_t
\begin{picture}(0,0)%
\includegraphics{eubijection.pstex}%
\end{picture}%
\setlength{\unitlength}{1973sp}%
\begingroup\makeatletter\ifx\SetFigFont\undefined%
\gdef\SetFigFont#1#2#3#4#5{%
  \reset@font\fontsize{#1}{#2pt}%
  \fontfamily{#3}\fontseries{#4}\fontshape{#5}%
  \selectfont}%
\fi\endgroup%
\begin{picture}(12349,4605)(61,-5626)
\end{picture}%

%% file: graphkodd.pstex_t
\begin{picture}(0,0)%
\includegraphics{graphkodd.pstex}%
\end{picture}%
\setlength{\unitlength}{1973sp}%
\begingroup\makeatletter\ifx\SetFigFont\undefined%
\gdef\SetFigFont#1#2#3#4#5{%
  \reset@font\fontsize{#1}{#2pt}%
  \fontfamily{#3}\fontseries{#4}\fontshape{#5}%
  \selectfont}%
\fi\endgroup%
\begin{picture}(3616,4536)(4493,-5707)
\put(5101,-1711){\makebox(0,0)[lb]{\smash{{\SetFigFont{10}{12.0}{\rmdefault}{\mddefault}{\updefault}{\color[rgb]{0,0,0}$n$}%
}}}}
\put(5926,-5611){\makebox(0,0)[lb]{\smash{{\SetFigFont{10}{12.0}{\rmdefault}{\mddefault}{\updefault}{\color[rgb]{0,0,0}$n'+1$}%
}}}}
\put(7126,-5311){\makebox(0,0)[lb]{\smash{{\SetFigFont{10}{12.0}{\rmdefault}{\mddefault}{\updefault}{\color[rgb]{0,0,0}$n'$}%
}}}}
\put(6451,-3436){\makebox(0,0)[lb]{\smash{{\SetFigFont{10}{12.0}{\rmdefault}{\mddefault}{\updefault}{\color[rgb]{0,0,0}$k'$ edges}%
}}}}
\end{picture}%

%% file: graphkodd2.pstex_t
\begin{picture}(0,0)%
\includegraphics{graphkodd2.pstex}%
\end{picture}%
\setlength{\unitlength}{1973sp}%
\begingroup\makeatletter\ifx\SetFigFont\undefined%
\gdef\SetFigFont#1#2#3#4#5{%
  \reset@font\fontsize{#1}{#2pt}%
  \fontfamily{#3}\fontseries{#4}\fontshape{#5}%
  \selectfont}%
\fi\endgroup%
\begin{picture}(3804,4530)(4261,-5626)
\put(4276,-2011){\makebox(0,0)[lb]{\smash{{\SetFigFont{10}{12.0}{\rmdefault}{\mddefault}{\updefault}{\color[rgb]{0,0,0}$n'+1$}%
}}}}
\put(5626,-3361){\makebox(0,0)[lb]{\smash{{\SetFigFont{10}{12.0}{\rmdefault}{\mddefault}{\updefault}{\color[rgb]{0,0,0}$k'$ edges}%
}}}}
\end{picture}%

%% file: graphkeven.pstex_t
\begin{picture}(0,0)%
\includegraphics{graphkeven.pstex}%
\end{picture}%
\setlength{\unitlength}{1973sp}%
\begingroup\makeatletter\ifx\SetFigFont\undefined%
\gdef\SetFigFont#1#2#3#4#5{%
  \reset@font\fontsize{#1}{#2pt}%
  \fontfamily{#3}\fontseries{#4}\fontshape{#5}%
  \selectfont}%
\fi\endgroup%
\begin{picture}(4288,5136)(3857,-5332)
\put(5101,-586){\makebox(0,0)[lb]{\smash{{\SetFigFont{10}{12.0}{\rmdefault}{\mddefault}{\updefault}{\color[rgb]{0,0,0}$2n$}%
}}}}
\put(6826,-736){\makebox(0,0)[lb]{\smash{{\SetFigFont{10}{12.0}{\rmdefault}{\mddefault}{\updefault}{\color[rgb]{0,0,0}$v_1$}%
}}}}
\put(7801,-1561){\makebox(0,0)[lb]{\smash{{\SetFigFont{10}{12.0}{\rmdefault}{\mddefault}{\updefault}{\color[rgb]{0,0,0}$v_2$}%
}}}}
\put(6901,-4936){\makebox(0,0)[lb]{\smash{{\SetFigFont{10}{12.0}{\rmdefault}{\mddefault}{\updefault}{\color[rgb]{0,0,0}$v_m$}%
}}}}
\put(6376,-5236){\makebox(0,0)[lb]{\smash{{\SetFigFont{10}{12.0}{\rmdefault}{\mddefault}{\updefault}{\color[rgb]{0,0,0}$n$}%
}}}}
\end{picture}%

%% file: bijection.pstex_t
\begin{picture}(0,0)%
\includegraphics{bijection.pstex}%
\end{picture}%
\setlength{\unitlength}{1973sp}%
\begingroup\makeatletter\ifx\SetFigFont\undefined%
\gdef\SetFigFont#1#2#3#4#5{%
  \reset@font\fontsize{#1}{#2pt}%
  \fontfamily{#3}\fontseries{#4}\fontshape{#5}%
  \selectfont}%
\fi\endgroup%
\begin{picture}(10230,4693)(1336,-5186)
\put(1351,-2836){\makebox(0,0)[lb]{\smash{{\SetFigFont{10}{12.0}{\rmdefault}{\mddefault}{\updefault}{\color[rgb]{0,0,0}$2_3$}%
}}}}
\put(3376,-736){\makebox(0,0)[lb]{\smash{{\SetFigFont{10}{12.0}{\rmdefault}{\mddefault}{\updefault}{\color[rgb]{0,0,0}$1_1$}%
}}}}
\put(4501,-1036){\makebox(0,0)[lb]{\smash{{\SetFigFont{10}{12.0}{\rmdefault}{\mddefault}{\updefault}{\color[rgb]{0,0,0}$2_1$}%
}}}}
\put(5251,-1786){\makebox(0,0)[lb]{\smash{{\SetFigFont{10}{12.0}{\rmdefault}{\mddefault}{\updefault}{\color[rgb]{0,0,0}$3_1$}%
}}}}
\put(5476,-2836){\makebox(0,0)[lb]{\smash{{\SetFigFont{10}{12.0}{\rmdefault}{\mddefault}{\updefault}{\color[rgb]{0,0,0}$4_1$}%
}}}}
\put(5176,-3961){\makebox(0,0)[lb]{\smash{{\SetFigFont{10}{12.0}{\rmdefault}{\mddefault}{\updefault}{\color[rgb]{0,0,0}$1_2$}%
}}}}
\put(4576,-4636){\makebox(0,0)[lb]{\smash{{\SetFigFont{10}{12.0}{\rmdefault}{\mddefault}{\updefault}{\color[rgb]{0,0,0}$2_2$}%
}}}}
\put(3451,-4936){\makebox(0,0)[lb]{\smash{{\SetFigFont{10}{12.0}{\rmdefault}{\mddefault}{\updefault}{\color[rgb]{0,0,0}$3_2$}%
}}}}
\put(2476,-4636){\makebox(0,0)[lb]{\smash{{\SetFigFont{10}{12.0}{\rmdefault}{\mddefault}{\updefault}{\color[rgb]{0,0,0}$4_2$}%
}}}}
\put(2251,-1036){\makebox(0,0)[lb]{\smash{{\SetFigFont{10}{12.0}{\rmdefault}{\mddefault}{\updefault}{\color[rgb]{0,0,0}$4_3$}%
}}}}
\put(9376,-811){\makebox(0,0)[lb]{\smash{{\SetFigFont{10}{12.0}{\rmdefault}{\mddefault}{\updefault}{\color[rgb]{0,0,0}$1_1$}%
}}}}
\put(9376,-5086){\makebox(0,0)[lb]{\smash{{\SetFigFont{10}{12.0}{\rmdefault}{\mddefault}{\updefault}{\color[rgb]{0,0,0}$1_2$}%
}}}}
\put(10876,-1411){\makebox(0,0)[lb]{\smash{{\SetFigFont{10}{12.0}{\rmdefault}{\mddefault}{\updefault}{\color[rgb]{0,0,0}$2_1$}%
}}}}
\put(11551,-2836){\makebox(0,0)[lb]{\smash{{\SetFigFont{10}{12.0}{\rmdefault}{\mddefault}{\updefault}{\color[rgb]{0,0,0}$3_1$}%
}}}}
\put(11026,-4411){\makebox(0,0)[lb]{\smash{{\SetFigFont{10}{12.0}{\rmdefault}{\mddefault}{\updefault}{\color[rgb]{0,0,0}$4_1$}%
}}}}
\put(7951,-1411){\makebox(0,0)[lb]{\smash{{\SetFigFont{10}{12.0}{\rmdefault}{\mddefault}{\updefault}{\color[rgb]{0,0,0}$4_2$}%
}}}}
\put(7951,-4411){\makebox(0,0)[lb]{\smash{{\SetFigFont{10}{12.0}{\rmdefault}{\mddefault}{\updefault}{\color[rgb]{0,0,0}$2_2$}%
}}}}
\put(7351,-2836){\makebox(0,0)[lb]{\smash{{\SetFigFont{10}{12.0}{\rmdefault}{\mddefault}{\updefault}{\color[rgb]{0,0,0}$3_2$}%
}}}}
\put(1651,-3961){\makebox(0,0)[lb]{\smash{{\SetFigFont{10}{12.0}{\rmdefault}{\mddefault}{\updefault}{\color[rgb]{0,0,0}$1_3$}%
}}}}
\put(1651,-1786){\makebox(0,0)[lb]{\smash{{\SetFigFont{10}{12.0}{\rmdefault}{\mddefault}{\updefault}{\color[rgb]{0,0,0}$3_3$}%
}}}}
\end{picture}%

%% file: CSP.bbl
\begin{thebibliography}{5}%%%%%%%%%%%%%%%%%%%%%%%%%
%%%%%%%%%%%%%%%%%%%%%%%%%%%%%%%%%%%%%%%%%%%

\bibitem{Eu}
S.-P. Eu, personal communication to V. Reiner and D. Stanton, April 2006.

\bibitem{EuFu}
S.-P. Eu, T.-S. Fu, The cyclic sieving phenomenon for faces of generalized cluster
complexes, \textit{Adv. in Appl. Math.} \textbf{40} (2008), no. 3, 350--376.

\bibitem{flajoletnoy}
P. Flajolet, M. Noy, Analytic combinatorics of non-crossing configurations,
\textit{Discrete Mathematics} \textbf{204} (1999), no. 1-3, 203--229.

\bibitem{RSW}
V. Reiner, D. Stanton, D. White, The cyclic sieving phenomenon,
\textit{J. Combin. Theory Ser. A} \textbf{108} (2004), 17--50.

%%%%%%%%%%%%%%%%%%%%%%%%%%%%%%%%%%%%%%%%%%%
\end{thebibliography}
